\documentclass[11pt,reqno]{amsart}
\usepackage{}
\usepackage{dsfont}
\usepackage{amsmath}
\usepackage{mathrsfs}
\usepackage{amsmath,amssymb}
\usepackage{amsfonts}
\usepackage{hyperref}
\usepackage{amsthm}
\usepackage{graphicx}
\theoremstyle{definition}
\newtheorem{lemma}{Lemma}[section]
\newtheorem{definition}[lemma]{Definition}
\newtheorem{proposition}[lemma]{Proposition}
\newtheorem{theorem}[lemma]{Theorem}
\newtheorem{corollary}[lemma]{Corollary}

\newtheorem{remark}{Remark}

\numberwithin{equation}{section}

\def\qed{\hfill{Q.E.D.}\smallskip}
\def\ac{{\textit {\textbf{Acknowledgement:}} }}
\title{Combinatorial Calabi Flows on Surfaces}
\author{Huabin Ge}
\date{}

\begin{document}
\maketitle

\begin{abstract}
For triangulated surfaces, we introduce the combinatorial Calabi flow which is an analogue of smooth Calabi flow. We prove that the solution of combinatorial Calabi flow exists for all time. Moreover, the solution converges if and only if Thurston's circle packing exists. As a consequence, combinatorial Calabi flow provides a new algorithm to find circle packings with prescribed curvatures. The proofs rely on careful analysis of combinatorial Calabi energy, combinatorial Ricci potential and discrete dual-Laplacians.
\end{abstract}
\section{Introduction}
In seeking constant curvature metric, E.Calabi studied the variational problem of minimizing the so-called ``Calabi energy" in any fixed cohomology class of K$\ddot{a}$hler metrics and proposed the Calabi flow (\cite{CA1}, \cite{CA2}). Hamilton introduced the Ricci flow technique (\cite{Ha1}), which has been used to solve the Poincar$\acute{e}$ conjecture. As to dimension two, \emph{i.e.}, smooth surface case, people proved that both Calabi flow and the normalized Ricci flow exist for all time $t\geq 0$. Moreover, these two flows converge to constant scalar curvature metric (see \cite{CHA1}, \cite{CHA2}, \cite{CHE}, \cite{CH1}, \cite{CHR}, \cite{Si}, and \cite{St}).

Given a triangulated surface, Thurston introduced the circle packing metric, which is a type of piecewise flat cone metric with singularities at the vertices. Thurston found that there are combinatorial obstructions for the existence of constant combinatorial curvature circle packing metric (see section 13.7 in \cite{T1}). In \cite{CL1}, Bennett Chow and Feng Luo defined combinatorial Ricci flow, which is the analogue of Hamilton's Ricci flow. Bennett Chow and Feng Luo proved that the combinatorial Ricci flow exists for all time and converges exponentially fast to Thurston's circle packing on surfaces. They reproved the equivalence between Thurston's combinatorial condition (see (1.3) in \cite {CL1}, or (\ref{combinatorial-condition}) in this paper) and the existence of constant curvature circle packing metric.

Inspired by the work of Bennett Chow and Feng Luo in \cite{CL1}, we consider in this paper the 2-dimensional combinatorial Calabi flow equation, which is the negative gradient flow of combinatorial Calabi energy. We interpret the Jacobian of the curvature map as a type of discrete Laplace operator, which comes from the dual structure of circle patterns. We evaluate an uniform bound for discrete dual-Laplacians, which implies the long time existence of the solutions of combinatorial Calabi flow. We prove that the combinatorial Calabi flow converges exponentially fast if and only if Thurston's combinatorial conditions are satisfied. The combinatorial Calabi flow tends to find Thurston's circle patterns automatically. As a consequence, we can design algorithm to compute circle packing metrics with prescribed combinatorial curvature. In fact, any algorithm minimizing combinatorial Calabi energy or combinatorial Ricci potential can achieve this goal.

\subsection{Circle packing metric}
Let's consider a closed surface $X$. A triangulation $T=(V,E,F)$ has a collection of vertices (denoted $V$), edges (denoted $E$), and faces (denoted $F$). A positive function $r:V\rightarrow (0,+\infty)$ defined on the vertices is called a circle packing metric and a function $\Phi: E\rightarrow [0, \pi/2]$ is called a weight on the triangulation. Throughout this paper, a function defined on vertices is regarded as a $N$-dimensional column vector and $N=V^{\#}$ is used to denote the number of vertices. Moreover, all vertices, marked by $v_{1},...,v_{N}$, are supposed to be ordered one by one and we often write $i$ instead of $v_i$. Thus we may think of circle packing metrics as points in $\mathds{R}^N_{>0}$, $N$ times of Cartesian product of $(0,\infty)$. A triangulated surface with wight $\Phi$ is often denoted as $(X,T,\Phi)$.

For fixed $(X,T,\Phi)$, every circle packing metric $r$ determines a piecewise linear metric on $X$ by attaching each edge $e_{ij}$ to a length $$l_{ij}=\sqrt{r^2_i+r^2_j+2r_ir_jcos(\Phi_{ij})}.$$ This length structure makes each face in $F$ isometric to an Euclidean triangle, more specifically, we can realize each face $\{i,j,k\}\in F$ as an Euclidean triangle with edge lengths $l_{ij},l_{jk},l_{ki}$ because the three positive numbers $l_{ij},l_{jk},l_{ki}$ (derived from $r_i$, $r_j$, $r_k$) must satisfy triangle inequalities (\cite{T1}, Lemma 13.7.2). Furthermore, the triangulated surface $(X,T)$ is composed by gluing many Euclidean triangles coherently.

\begin{figure}
\begin{minipage}[t]{0.48\linewidth}
\centering
\includegraphics[width=0.75\textwidth]{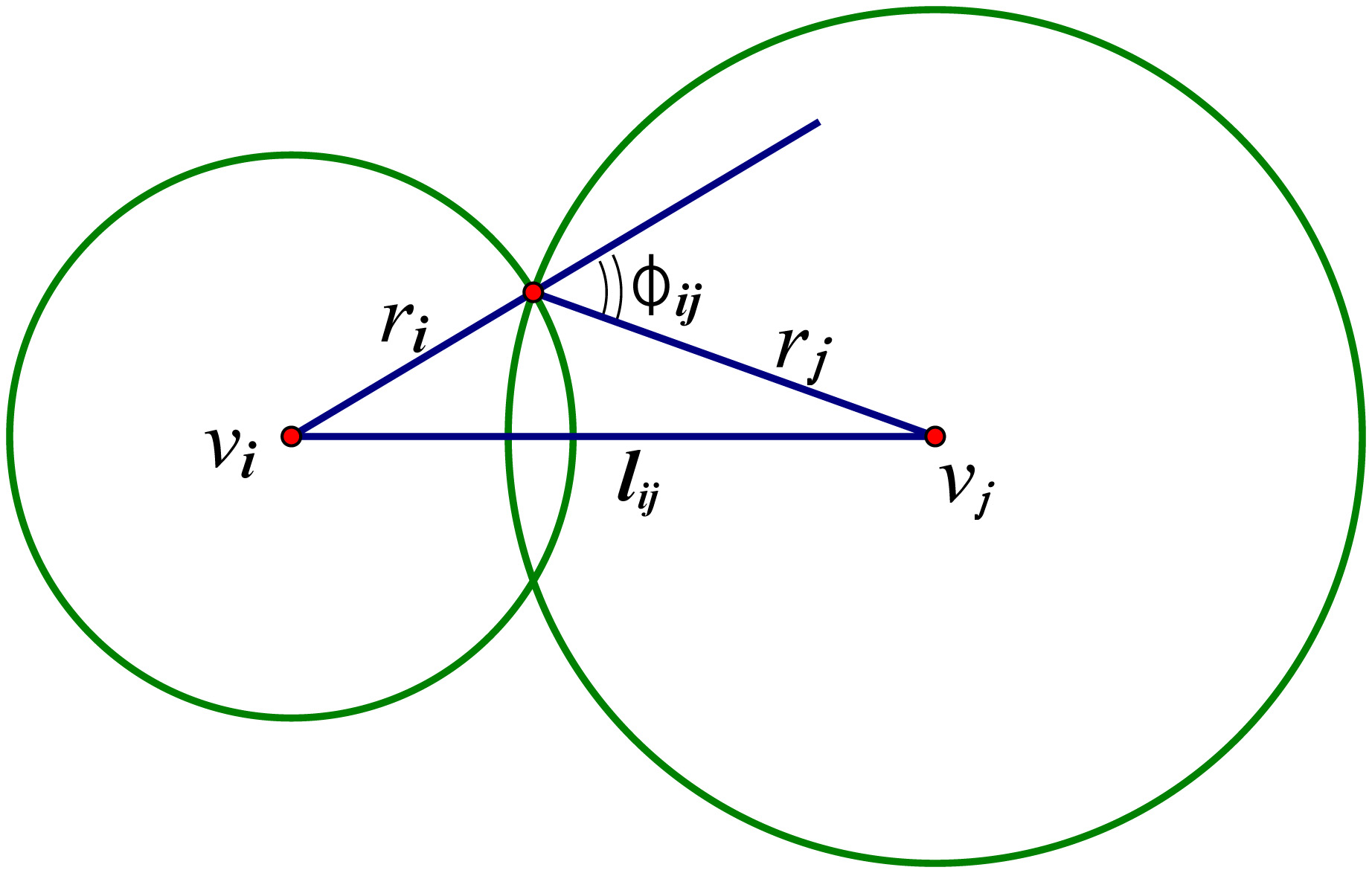}
\caption{circle packing metric}
\end{minipage}
\begin{minipage}[t]{0.5\linewidth}
\centering
\includegraphics[width=0.5\textwidth]{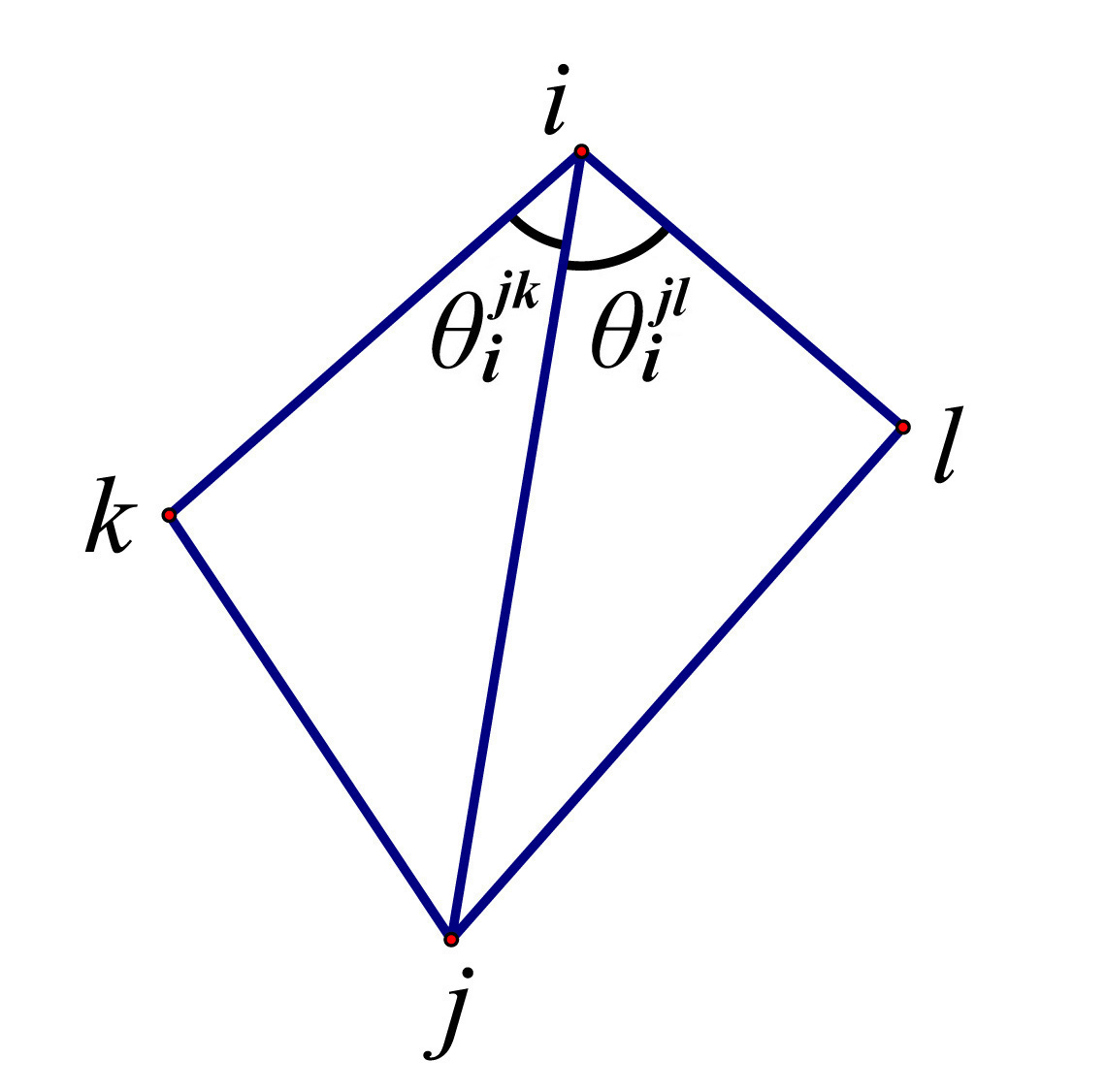}
\caption{two adjacent triangles}
\end{minipage}
\end{figure}

\subsection{Combinatorial curvature and constant curvature metric}
Given $(X,T,\Phi)$ and a circle packing metric $r$, all inner angles of the triangles are determined by $r_1,\cdots,r_N$. Denote $\theta^{jk}_i$ as the inner angle at vertex $i$ in the triangle $\{i,j,k\}\in F$, then the well known combinatorial (or ``discrete") Gauss curvature $K_i$ at vertex $i$ is
\begin{equation}
K_i=2\pi-\sum_{\{i,j,k\}\in F}\theta^{jk}_i
\end{equation}
where the sum is taken over each triangle having $i$ as one of its vertices. Notice that $\theta_i^{jk}$ can be calculated by cosine law, thus $\theta_i^{jk}$ and $K_i$ are elementary functions of the circle packing metric $r$.

For every circle packing metric $r$ on $(X,T,\Phi)$, we have the combinatorial Gauss-Bonnet formula (\cite {CL1})
\begin{equation}\label{Gauss-Bonnet formula}
\sum_{i=1}^NK_i=2\pi \chi(X).
\end{equation}
Notice that, above combinatorial Gauss-Bonnet formula (\ref{Gauss-Bonnet formula}) is still valid for all piecewise linear surface $(X,T,l)$, where $l:E\rightarrow(0,+\infty)$ is a piecewise linear metric, that is, the length structure $l$ makes each face in $F$ isometric to an Euclidean triangle. The average combinatorial curvature is $k_{av}=2\pi \chi(X)/N$. It will not change as the metric (circle packing metric or piecewise linear metric) varies, and it's a invariant that depends only on the topological (the Euler characteristic number $\chi(X)$) and combinatorial ($N=V^{\#}$) information of $(X,T)$.

Finding good metrics is always a central topic in geometry and topology. The constant curvature circle packing metric (denoted as $r_{av}$), a metric that determines constant combinatorial curvature $K_{av}=K(r_{av})=k_{av}(1,\cdots,1)^T$, seems a good candidate for privileged metric. Thurston first studied this class of metrics, and found that there are combinatorial obstructions for the existence of constant curvature metric (\cite{T1}). We quote Thurston's combinatorial conditions here.
\begin{lemma}
Given $(X,T,\Phi)$, where $X$, $T$ and $\Phi$ are defined as before. Then there exists constant combinatorial curvature circle packing metric if and only if the following combinatorial conditions are satisfied
\begin{equation} \label{combinatorial-condition}
2\pi\chi(X)\frac{|I|}{|V|}>-\sum_{(e,v)\in Lk(I)}(\pi-\Phi(e))+2\pi\chi(F_I)\,,\:\:\:\:\: \forall I: \phi \neq I \subsetneqq V,
\end{equation}
where $F_I$ is the subcomplex whose vertices are in $I$, $Lk(I)$ is the set of pairs $(e,v)$ of an edge $e$ and a vertex $v$ so that (1) the end points of $e$ are not in $I$ and (2) the vertex $v$ is in $I$ and (3) $e$ and $v$ form a face in $T$.
\end{lemma}

In \cite{CL1}, Bennett Chow and Feng Luo introduced the combinatorial Ricci flow which is an analogue of Hamilton's Ricci flow on surfaces (\cite{Ha}). They proved that the solutions of the normalized combinatorial Ricci flow converges exponentially fast to Thurston's constant curvature metric $r_{av}$. Combinatorial Ricci flow method is so powerful that, on one hand, the flow provides a new proof of Thurston's existence of constant circle packing metric theorem, on the other hand, the flow suggests a new algorithm to find circle packing metrics with prescribed curvatures.

In this paper, we introduce the combinatorial Calabi flow, which is the negative gradient flow of combinatorial Calabi energy. Combinatorial Calabi flow is an analogue of smooth Calabi flow on surface. We prove that the solution of combinatorial Calabi flow exists for $t\in[0,+\infty)$ by a careful estimation of discrete dual-Laplacian. Moreover, the solution of combinatorial Calabi flow converges to constant curvature metric $r_{av}$ if and only if $r_{av}$ exists.

\section{The 2-Dimensional Combinatorial Calabi Flow}
\subsection{Definition of combinatorial Calabi flow}
For smooth surface, the well known Calabi flow is $ \frac{\partial g}{\partial t}= \Delta K g$, where \emph{K} is Gauss curvature. Obviously, smooth Laplace-Beltrami operator is very important for smooth Calabi flow equation. Before giving the definition of combinatorial Calabi flow, we first need to define combinatorial Laplace operator, which is an analogue of smooth Laplace-Beltrami operator.

Set $u_i=\ln r_i$, $i=1,\cdots,N$. Then the coordinate transformation $u=u(r)$ maps $r\in\mathds{R}^N_{>0}$ to $u\in\mathds{R}^N$ homeomorphically (denote the inverse coordinate transformation as $Exp:\mathds{R}^N\rightarrow\mathds{R}^N_{>0},\,u\mapsto r$ for latter use). Similar to \cite{Ge1} and \cite{Ge2}, we interpret the discrete Laplacian to be the Jacobian of the curvature map $K=K(u)$.
\begin{definition}
Given $(X,T,\Phi)$, where $X$, $T$ and $\Phi$ is defined as before. The discrete dual-Laplace operator $``\Delta"$ is defined to be $\Delta=-L^T$, where
\begin{equation*}
  L=(L_{ij})_{N\times N}=\frac{\partial(K_1,...,K_N)}{\partial(u_1,...,u_N)}=
  \left(
\begin{array}{ccccc}
 {\frac{\partial K_1}{\partial u_1}}& \cdot & \cdot & \cdot &  {\frac{\partial K_1}{\partial u_N}} \\
 \cdot & \cdot & \cdot & \cdot & \cdot \\
 \cdot & \cdot & \cdot & \cdot & \cdot \\
 \cdot & \cdot & \cdot & \cdot & \cdot \\
 {\frac{\partial K_N}{\partial u_1}}& \cdot & \cdot & \cdot &  {\frac{\partial K_N}{\partial u_N}}
\end{array}
\right).
\end{equation*}
\end{definition}
Both $\Delta$ and $L$ acts on functions $f$ (defined on vertices, hence is a column vector) by matrix multiplication, \emph{i.e.}
\begin{equation*}
\Delta f_i=(\Delta f)_i=-(L^Tf)_i=-\sum_{j=1}^N\frac{\partial K_j}{\partial u_i}f_j=-\sum_{j=1}^N\frac{\partial K_j}{\partial r_i}r_if_j.
\end{equation*}
Discrete Laplacian defined here is exactly the famous discrete dual-Laplacian (\cite{G1}\cite{G3}), which will be explained carefully in Appendix \ref{Appendix-DDL}.
\begin{definition}
Given data $(X,T,\Phi)$, where $X$, $T$ and $\Phi$ is defined as before. The combinatorial Calabi flow is
\begin{equation} \label{Calabi-flow-r}
\cfrac{dr_i}{dt}=r_i \Delta K_i
\end{equation}
with initial circle packing metric $r(0)\in \mathds{R}^N_{>0}$, or equivalently,
\begin{equation} \label{Calabi-flow-u}
\cfrac{du_i}{dt}=\Delta K_i
\end{equation}
with $u(0)\in \mathds{R}^N$.
\end{definition}
It's more neatly to write combinatorial Calabi flow (\ref{Calabi-flow-r}) and (\ref{Calabi-flow-u}) in matrix forms. Denote $R=diag\{r_1,\cdots,r_N\}$, then flow (\ref{Calabi-flow-r}) turns to
\begin{equation} \label{Calabi-flow-r-matrix}
 \dot{r}=R\Delta K=-RL^TK,
\end{equation}
where the symbol $``\cdot"$ represents the time derivative. The flow (\ref{Calabi-flow-u}) turns to
\begin{equation} \label{Calabi-flow-u-matrix}
 \dot{u}=\Delta K=-L^TK.
\end{equation}
Notice that, equation (\ref{Calabi-flow-u-matrix}) is an autonomous ODE system.

\begin{remark}
Using the matrix language, Bennett Chow and Feng Luo's combinatorial Ricci flow (\cite{CL1}) is $\dot{u}=-K$. The normalized combinatorial Ricci flow is $\dot{u}=K_{av}-K$.
\end{remark}

\subsection{Main properties of combinatorial Calabi flow}
\begin{definition}
For data $(X,T,\Phi)$ with a circle packing metric $r$, the combinatorial Calabi energy is
\begin{equation}\label{Calabi-energy}
\mathcal{C}(r)=\|K-K_{av}\|^2=\sum_{i=1}^N (K_i-k_{av})^2
\end{equation}
\end{definition}

Consider the combinatorial Calabi energy $\mathcal{C}$ as a function of $u$, we obtain
\begin{displaymath}
\nabla_u \mathcal{C}=
\left(
\begin{array}{c}
 {\frac{\partial \mathcal{C}}{\partial u_1}} \\
 \cdot\\
 \cdot\\
 \cdot\\
 {\frac{\partial \mathcal{C}}{\partial u_N}}
\end{array}
\right)
=2
\left(
\begin{array}{ccccc}
 {\frac{\partial K_1}{\partial u_1}}& \cdot & \cdot & \cdot &  {\frac{\partial K_N}{\partial u_1}} \\
 \cdot & \cdot & \cdot & \cdot & \cdot \\
 \cdot & \cdot & \cdot & \cdot & \cdot \\
 \cdot & \cdot & \cdot & \cdot & \cdot \\
 {\frac{\partial K_1}{\partial u_N}}& \cdot & \cdot & \cdot &  {\frac{\partial K_N}{\partial u_N}}
\end{array}
\right)
\left(
\begin{array}{c}
 K_1 \\
 \cdot\\
 \cdot\\
 \cdot\\
 K_N
\end{array}
\right)=2L^TK.
\end{displaymath}

\begin{proposition}
Combinatorial Calabi flow (\ref{Calabi-flow-u}) (or (\ref{Calabi-flow-u-matrix})) is the negative gradient flow of combinatorial Calabi energy. Moreover, the calabi energy (\ref{Calabi-energy}) is descending along this flow.
\end{proposition}
\begin{proof}
$\dot{u}(t)=-L^TK=-\frac{1}{2}\nabla_u \mathcal{C}$. Moreover, $\mathcal{C}'(t)=-\frac{1}{2}\|\nabla_u \mathcal{C}\|^2\leq0$.
\end{proof}\qed

Because $K_i$ and $\Delta K_i$ are elementary functions of $r_1,...,r_N$, the local existence of combinatorial Calabi flow (\ref{Calabi-flow-r}) follows from Picard's existence and uniqueness theorem in the standard ODE theory. By a careful estimation of discrete dual-Laplacian and combinatorial Ricci potential, the long time existence property is proved in section \ref{Long-existence-section}, the convergence property is proved in section \ref{convergence-section}. The main results in this paper is stated as follows
\begin{theorem}
Fix $(X,T,\Phi)$, where $X$ is a closed surface, $T$ is a triangulation, $\Phi\in[0,\pi/2]$ is a weight. For any initial circle packing metric $r(0)\in \mathds{R}^N_{>0}$, the solution of combinatorial Calabi flow (\ref{Calabi-flow-r}) exists for $t\in [0,\,+\infty)$. Moreover, the following four statements are mutually equivalent
\begin{description}
  \item[(1)] The solution of combinatorial Calabi flow (\ref{Calabi-flow-r}) converges.
  \item[(2)] The solution of combinatorial Ricci flow $\dot{u}=K_{av}-K$ converges.
  \item[(3)] There exists constant curvature circle packing metric $r_{av}$.
  \item[(4)] $2\pi\chi(X)\frac{|I|}{|V|}>-\sum_{(e,v)\in Lk(I)}(\pi-\Phi(e))+2\pi\chi(F_I)\,,\:\:\:\:\: \forall I: \phi \neq I \subsetneqq V.$
\end{description}
Furthermore, if any of above four statements is true, the solution of combinatorial Calabi flow (\ref{Calabi-flow-r}) converges exponentially fast to constant curvature circle packing metric $r_{av}$.
\end{theorem}

\section{Long Time Existence}\label{Long-existence-section}
\subsection{Discrete Laplacian for circle packing metrics}\label{Discrete-Dual-Laplac}
Any circle packing metric $r$ determines an intrinsic metric structure on fixed $(X,T,\,\Phi)$ by Euclidean cosine law. The lengths $l_{ij}$, angles $\theta_i^{jk}$ and curvatures $K_i$ are elementary functions of $r=(r_1, ... ,r_N)^T$. We write $j\sim i$ in the following if the vertices $i$ and $j$ are adjacent. For any vertex $i$ and any edge $j\sim i$, set
\begin{equation}
B_{ij}=\frac{\partial(\theta_i^{jk}+\theta_i^{jl})}{\partial r_j} r_j,
\end{equation}
then $B_{ij}=B_{ji}$, since $\frac{\partial\theta_i^{jk}}{\partial r_j} r_j=\frac{\partial\theta_j^{ik}}{\partial r_i} r_i$ (see Lemma 2.3 in \cite{CL1}).

\begin{proposition}
For any $1\leq i,j\leq N$ and $i\sim j$, we have
\begin{equation} \label{Bij}
 0<B_{ij}<2\sqrt{3}.
\end{equation}
\end{proposition}
\begin{proof}
We just need to prove $0<\frac{\partial \theta^{jk}_i}{\partial r_j} r_j<\sqrt{3}$. We defer the details to the appendix.

\qed
\end{proof}

\begin{proposition} \label{Lii=sum-Bij}
For any $1\leq i,j\leq N$,
\begin{gather}
L_{ij}=
\begin{cases}
\;\sum\limits_{k \sim i}B_{ik} \,, & \text{$j=i$} \\
\;\;-B_{ij} \,,& \text{$j\sim i$}\\
\;\;\;\;\;0 \,,& \text{$else.$}
\end{cases}
\end{gather}
\end{proposition}
\begin{proof}
This can be proved by direct calculations and we omit the details.
\end{proof}\qed

\begin{proposition}\label{L-semi-positive-rank-ker}
$L$ is semi-positive definite, having rank $N-1$. Moreover, the kernel of $L$ is the span of the vector $(1,\cdots,1)^T$.
\end{proposition}
\begin{proof}
This is a direct consequence from Lemma 3.10 in \cite{CL1}.\qed
\end{proof}

The differential form $\omega\triangleq\sum_{i=1}^N(K_i-k_{av})du_i$ is closed, for $L_{ij}=L_{ji}$. Thus the integral $\int_{u_0}^u\sum_{i=1}^N\big(K_i-k_{av}\big)du_i$ makes good sense (\cite{CL1}), where $u_0$ is an arbitrary point in $\mathds{R}^N$. This integral is crucial for the proof of our main theorem. For convenience, we call this integral combinatorial Ricci potential.

For any smooth closed manifold $(M,g)$ with Riemannian metric $g$, we know that $\int_M\Delta f=0$, where $f$ is an arbitrary smooth function defined on $M$. For combinatorial surface $(X,T,\Phi)$ with circle packing metric $r$, we have $\sum\limits_{i=1}^N\Delta K_i=-(1,\cdots,1)^TLK=0$. Thus we obtain

\begin{proposition} \label{r1...rN=constant}
As long as the combinatorial Calabi flow exists, both $ \prod_{i=1}^Nr_i(t)\equiv \prod_{i=1}^Nr_i(0)$ and
$ \sum_{i=1}^Nu_i(t)\equiv\sum_{i=1}^Nu_i(0)$ are constants.
\end{proposition}

\subsection{Long time existence of combinatorial Calabi flow}
Notice that
\begin{equation}
 \Delta K_i=\sum_{j\sim i}B_{ij}(K_j-K_i).
\end{equation}
Using the estimation of $B_{ij}$ in (\ref{Bij}), we obtain
\begin{theorem} \label{long-time-exist}
Given $X,T,\Phi$, where $X$, $T$ and $\Phi$ are defined as before. For any initial circle packing metric $r(0)$, the solution of the combinatorial Calabi flow (\ref{Calabi-flow-r}) exists for all time $t\in[0,+\infty)$.
\end{theorem}
\begin{proof}
Let $d_i$ denote the degree (or say valence) at vertex $v_i$, that is the number of edges adjacent to $v_i$. Suppose $d=max(d_1\,,...\,,d_N)$, evidently we see $(2-d)\pi<K_i<2\pi$. From the evaluation of ${B_{ij}}$ in Lemma 1 we know, all $|\Delta K_i|$ is uniformly bounded by a positive constant $c=2\sqrt{3}\cdot N\cdot d\cdot \pi$ which depend only on the triangulation. Then we obtain
$$ c_0e^{-ct}\leq r_i(t)\leq c_0e^{ct}$$
where $c_0=c\left(r(0)\right)$ is a positive constant depends only on the initial state of $r(t)$. This implies that the combinatorial Calabi flow has a solution for all time $t\in[0,\infty)$ for any initial value $r(0)\in \mathds{R}^N_{>0}$. \qed
\end{proof}
\begin{remark}
The long time existence of $r(t)$ can be deduced without using the uniform estimation of $B_{ij}$ in (\ref{Bij}). See Remark \ref{newproof-T-is-infty}.
\end{remark}

\section{Convergence to Constant Curvature Metric}\label{convergence-section}
\subsection{Evolution of combinatorial curvature}
We consider a general flow in the space of all circle packing metrics with form
$$ \frac{dr_i}{dt}=r_i \varphi_i $$
where $\varphi_i$ are smooth functions of $r_1,...,r_N$. For this flow, we have
\begin{lemma}\label{cur-evol}
Under general flow, the discrete Gauss curvature evolves according to
$$ \frac{dK_i}{dt}=-\Delta\varphi_i=(L\varphi)_i $$
or $ \dot{K}=L\varphi $ in matrix form.
\end{lemma}
\begin{proof}
We can calculate $\frac{d\theta^{jk}_i}{dt}$ and $\frac{dK_i}{dt}$ in components form to get the proof. The detail is tedious and is omitted. However, above Lemma is a direct conclusion of chain rule in matrix form
$$\dot{K}=\frac{\partial(K_1,\cdots,K_N)}{\partial(u_1,\cdots,u_N)}\dot{u}=L\varphi=-\Delta\varphi.$$ \qed
\end{proof}

As to combinatorial Calabi flow, select $\varphi_i=\Delta K_i$ (or $\varphi=-LK$), and use lemma \ref{cur-evol}, we obtain
\begin{proposition}
Under combinatorial Calabi flow, the discrete Gauss curvature evolves according to
\begin{equation}
\frac{dK}{dt}=-L^2K
\end{equation}
\end{proposition}

\subsection{Global rigidity and admissible space of curvatures}
For fixed data $(X,T,\Phi)$, let $\Pi:\mathds{R}_{>0}^N\rightarrow \mathds{R}^N$ be the curvature map which maps $r$ to $K$. For any positive constant $c$, denote
\begin{equation}
\mathscr{P}_c\triangleq\Big\{r=(r_1,\cdots,r_N)^T\in \mathds{R}^N_{>0}\,\Big|\,\prod_{i=1}^Nr_i=c\Big\}.
\end{equation}
Thurston proved that the map $\Pi_{\mathscr{P}_c}$ is injective, \emph{i.e.}, the metric is determined by its curvature up to a scalar multiplication (\cite{T1}, \cite{CL1}). This phenomenon is called ``\textbf{global rigidity}". Denote $$\mathscr{K}_{GB}\triangleq\Big\{K=(K_1,\cdots,K_N)^T\in \mathds{R}^N\;\Big|\;\sum_{i=1}^NK_i=2\pi\chi(X)\Big\}.$$
For any nonempty proper subset $I$ of vertices $V$, set
\begin{equation*}
\mathscr{Y}_I\triangleq\Big\{(K=K_1,\cdots,K_N)^T\in \mathds{R}^N \;\Big|\;\sum_{i\in I}^NK_i >-\sum_{(e,v)\in Lk(I)}(\pi-\Phi(e))+2\pi\chi(F_I)\,\Big\}
\end{equation*}
Let
\begin{equation}
\mathscr{Y}\triangleq\mathscr{K}_{GB} \cap \Big(\cap_{\phi\,\neq\,I \,\subsetneqq\,V} \mathscr{Y}_I \Big),
\end{equation}
Bennett Chow and Feng Luo proved that $\mathscr{Y}=\Pi\left(\mathds{R}_{>0}^N\right)$ (\cite{CL1}), so, $\mathscr{Y}$ is exactly the space of all possible combinatorial curvatures. We call $\mathscr{Y}$ ``\textbf{admissible space of curvatures}". Notice that $\mathscr{Y}$ is a nonempty bounded open convex polyhedron in $(N-1)$-dimensional hyperplane $\mathscr{K}_{GB}$. Moreover, $\mathscr{Y}$ is homeomorphic to $\mathds{R}^{N-1}$. By the invariance of domain theorem in differential topology, the curvature map $\Pi$ realized a homeomorphism from $\mathscr{P}_c$ to $\mathscr{Y}$. Denote
\begin{equation}
\mathscr{U}_a=\Big\{u=(u_1,...,u_N)^T\in \mathds{R}^N\,\Big|\,\sum_{i=1}^N u_i=a\Big\},
\end{equation}
then $\mathscr{P}_c=Exp(\mathscr{U}_a)$, where $a=\ln c$. Denote $\widetilde{\Pi}=\Pi\circ Exp:\mathds{R}^N\rightarrow \mathds{R}^N,u\mapsto K$. Then we have homeomorphisms $\mathscr{U}_a\xrightarrow[]{u\mapsto r=Exp(u)}\mathscr{P}_c \xrightarrow[]{r\mapsto K=\Pi(r)}\mathscr{Y}$ and, moreover, for arbitrary positive constant $c$, $\mathscr{Y}=\Pi\left(\mathds{R}_{>0}^N\right)=\Pi\left(\mathscr{P}_c\right)=
\widetilde{\Pi}\left(\mathds{R}^N\right)=\widetilde{\Pi}\left(\mathscr{U}_a\right)$, where $a=\ln c$.
\subsection{Convergence to constant curvature metric}
The constant curvature $K_{av}=k_{av}(1,\,\cdots,1)^T$ may not admissible, that is, $K_{av}$ may not belongs to $\mathscr{Y}$. $K_{av}\in \mathscr{Y}$ is equivalent to the existence of constant curvature metric $r_{av}$, actually, $\Pi^{-1}(K_{av})\subset \mathds{R}_{>0}^N$ is a class of constant curvature metrics which differ by scalar multiplications.

Along the combinatorial Calabi flow (\ref{Calabi-flow-r}), from Proposition \ref{r1...rN=constant} we know that $\{r(t)\}\subset\mathscr{P}_c$, where $c=\prod_{i=1}^N r_i(0)$. Now we are at the stage of proving convergence results. We state our main results as follows:
\begin{theorem}
Fix $(X,T,\Phi)$, where $X$ is a closed surface, $T$ is a triangulation, $\Phi\in[0,\pi/2]$ is a weight. For any initial circle packing metric $r(0)\in \mathds{R}^N_{>0}$, the solution of combinatorial Calabi flow (\ref{Calabi-flow-r}) exists for $t\in [0,\,+\infty)$. Moreover, the following four statements are mutually equivalent
\begin{description}
  \item[(1)] The solution of combinatorial Calabi flow (\ref{Calabi-flow-r}) converges.
  \item[(2)] The solution of combinatorial Ricci flow $\dot{u}=K_{av}-K$ converges.
  \item[(3)] There exists constant curvature circle packing metric $r_{av}$.
  \item[(4)] $2\pi\chi(X)\frac{|I|}{|V|}>-\sum_{(e,v)\in Lk(I)}(\pi-\Phi(e))+2\pi\chi(F_I)\,,\:\:\:\:\: \forall I: \phi \neq I \subsetneqq V.$
\end{description}
Furthermore, if any of above four statements is true, the solution of combinatorial Calabi flow (\ref{Calabi-flow-r}) converges exponentially fast to constant curvature circle packing metric $r_{av}$.
\end{theorem}
\begin{proof}
We have proved the long time existence of flow (\ref{Calabi-flow-r}) in last section. Denote $r(t)$, $t\in[0,\infty)$ as the solution of combinatorial Calabi flow (\ref{Calabi-flow-r}).

First we prove $(1)\Rightarrow(3)$.
If $r(t)$ converges, \emph{i.e.}, $r(+\infty)=\lim\limits_{t\rightarrow+\infty}r(t)\in \mathds{R}_{>0}^N$ exists, then both $K(+\infty)=\lim\limits_{t\rightarrow+\infty}K(t)\in \mathscr{Y}$ and $L(+\infty)=\lim\limits_{t\rightarrow+\infty}L(t)$ exist too. Consider the combinatorial Calabi energy
\begin{equation*}
\mathcal{C}(r(t))= \sum_{i=1}^N\big(K_i(t)-k_{av}\big)^2 = \sum_{i=1}^N K^2_i(t)-N k^2_{av},
\end{equation*}
differentiate over t, we obtain
\begin{equation*}
\mathcal{C'}(t)= 2\sum_{i=1}^N\dot{K}_iK_i = 2K^T\dot{K}=-2K^TL^2K\leq0,
\end{equation*}
which implies that $\mathcal{C}(t)$ is descending as time $t$ increases. Moreover, $\mathcal{C}(t)$ is bounded by zero from below. Therefore $\mathcal{C}(+\infty)=\lim\limits_{t\rightarrow+\infty} \mathcal{C}(t)$ exists.
$\mathcal{C'}(+\infty)$ exists too, since
 $$\mathcal{C'}(+\infty)= \lim\limits_{t\rightarrow+\infty}(-2K^T(t)L^2(t)K(t))=-2K^T(+\infty)L^2(+\infty)K(+\infty).$$
The existence of $\mathcal{C}(+\infty)$ and $\mathcal{C'}(+\infty)$ implies $\mathcal{C'}(+\infty)=0$, that is,
$$K^T(+\infty)L^2(+\infty)K(+\infty)=0.$$
Hence $K(+\infty)\in Ker(L^2)=Ker(L)$. From Lemma \ref{L-semi-positive-rank-ker} we know $K_{av}=K(+\infty)\in \mathscr{Y}$, which implies that constant curvature metric $r_{av}$ exists.

Next we prove $(3)\Rightarrow(1)$. Assuming there exists constant curvature circle packing metric $r_{av}$, which implies $K_{av}\in \mathscr{Y}$, we want to show $r(t)$ converges, \emph{i.e.} $r(+\infty)=\lim\limits_{t\rightarrow+\infty}r(t)\in \mathds{R}_{>0}^N$. We carry out the proof in three steps.

Step 1: Denote $\lambda_1$ as the first eigenvalue of $L$, we claim that
\begin{equation} \label{K^TLK>lamda.C}
K^TL^2K\geq \lambda_1^2\|K-K_{av}\|^2=\lambda_1^2 \mathcal{C}.
\end{equation}
Select an orthogonal matrix $P$, such that $P^TLP=diag\{0,\lambda_1,\cdots,\lambda_{N-1}\}$. Write $P=(e_0,e_1,\cdots,e_{N-1})$, where $e_i$ is the $(i+1)^{th}$ column of $P$. Then $Le_0=0$ and $Le_i=\lambda_ie_i$, $i=1,\cdots,N-1$, which implies $e_0=\frac{1}{\sqrt{N}}(1,\cdots,1)^T$. Write $\widetilde{K}=P^TK$, then $\widetilde{K}_1=e_0^TK=\sqrt{N}k_{av}$. Hence we obtain
$$K^TL^2K=\lambda_1^2\widetilde{K}_2^2+\cdots+\lambda_{N-1}^2\widetilde{K}_N^2\geq\lambda_1^2(\|\widetilde{K}\|^2-\widetilde{K}_1^2)
=\lambda_1^2\|K-K_{av}\|^2=\lambda_1^2 \mathcal{C}.$$
Thus we get the claim above.

Step 2: we show that $\{r(t)\,|\,t\in[0,\infty)\}\subset\subset\mathds{R}^N_{>0}$. Consider the combinatorial Ricci potential
\begin{equation}
f(u)\triangleq\int_{u_{av}}^u\sum_{i=1}^N\big(K_i-k_{av}\big)du_i\,,\,\,u\in \mathds{R}^N,
\end{equation}
where $u_{av}=Exp^{-1}(r_{av})$. This integral is well defined, because $\sum_{i=1}^N(K_i-k_{av})du_i$ is a closed differential form. $f\big|_{\mathscr{U}_a}$ is strictly convex (see the proof of Theorem \ref{Ricci-potential-conveg-infty}), hence $\nabla f\big|_{\mathscr{U}_a}: \mathscr{U}_a\rightarrow \mathds{R}^N, u\mapsto K-K_{av}$ is injective, therefore $u_{av}$ is the unique critical point of $f\big|_{\mathscr{U}_a}$ and $f$ is bounded from blow by zero ($f(u)\geq f(u_{av})=0$). Consider $\varphi(t)\triangleq f\big(u(t)\big)$, then
$$\varphi'(t)=(\nabla f)^T \dot{u}(t)=(K-K_{av})^T (-LK)=-K^TLK\leq0$$
hence $\varphi(t)$ is descending as $t$ increases and then $0\leq\varphi(t)\leq \varphi(0)=f\big(u(0)\big)$, for any $t\in[0,+\infty)$. Hence $\big\{u(t)\,| \,t\in [0,+\infty)\big\} \subset \left(f\big|_{\mathscr{U}}\right)^{-1}\big([0,\varphi(0)]\big)$.
Because $f\big|_{\mathscr{U}}$ is proper by Theorem \ref{Ricci-potential-conveg-infty}, $\left(f\big|_{\mathscr{U}}\right)^{-1}\big([0,\varphi(0)]\big)$ is a compact subset  of $\mathscr{U}$. Therefore $\big\{u(t)\,|\; t\in [0,\,+\infty)\big\}\subset\subset\mathscr{U}$, or equivalently, $$\{r(t)\,|\,t\in[0,\infty)\}\subset\subset\mathscr{P}\subset\mathds{R}^N_{>0}.$$

Step 3: we show that $r(t)$ converges exponentially fast to $r_{av}$. Due to step 2, $\lambda_1^2(t)$, the first eigenvalue of $L^2(t)$, has a uniform lower bound along the Calabi flow, \emph{i.e.}, $\lambda_1^2(t)\geq \lambda/2>0$, where $\lambda$ is a positive constant. Using (\ref{K^TLK>lamda.C}), we obtain $$\mathcal{C'}(t)=-2K^TL^2K\leq -2\lambda_1^2(t) \mathcal{C}\leq-\lambda \mathcal{C}.$$
Hence then we have
\begin{equation}
\mathcal{C}(t)\leq \mathcal{C}(0)e^{-\lambda t},
\end{equation}
which implies that the curvature $K_i(t)$ converges exponentially fast to $k_{av}$. Moreover, we can deduce that $r(t)$ and $u(t)$ converge exponentially fast to $r_{av}$ and $u_{av}$ respectively by similar methods as in the proof of Lemma 4.1 in \cite{Ge2}.

The equivalence between (2), (3) and (4) can be found in \cite{CL1} (Thurston first claimed that (3) and (4) are equivalent in \cite{T1}). Thus we finished the proof.\qed
\end{proof}

\begin{remark} \label{newproof-T-is-infty}
We can prove $\{r(t)\,|\,t\in[0,T)\}\subset\subset\mathds{R}^N_{>0}$ similarly without assuming $T=+\infty$. Using the classical extension theorem of solutions in ODE theory, we obtain $T=+\infty$. This gives a new proof of Theorem \ref{long-time-exist}.
\end{remark}

\begin{theorem}
For $(X,T,\Phi)$, consider the following user prescribed combinatorial Calabi flow
\begin{equation}\label{user-prescribed-flow}
\dot{u}=L(\overline{K}-K),
\end{equation}
where $\overline{K}\in\mathds{R}^N$ is any prescribed curvature. For any initial circle packing metric $r(0)\in \mathds{R}^N_{>0}$, the solution of (\ref{user-prescribed-flow}) exists for $t\in [0,\,+\infty)$. Moreover, the following three statements are mutually equivalent
\begin{description}
  \item[(1)] The solution of (\ref{user-prescribed-flow}) converges.
  \item[(2)] The prescribed curvature $\overline{K}$ is admissible, \emph{i.e.}, $\overline{K}\in\mathscr{Y}$.
  \item[(3)] The solution of user prescribed Ricci flow $\dot{u}=\overline{K}-K$ converges.
\end{description}
Furthermore, if any of above three statements is true, the solution of flow (\ref{user-prescribed-flow}) converges exponentially fast to $\bar{r}$.
\end{theorem}

\ac
First I want to show my greatest respect to my advisor Gang Tian who brought me to the area of combinatorial curvature flows several years ago, he have been so kind, loving and caring to me in last few years. I like to give special thanks to Professor Feng Luo for reading the paper carefully and giving numerous improvements to the paper. I am also grateful to Dongfang Li who had done some numerical simulations for combinatorial Calabi flow. At last, I would like to thank Guanxiang Wang, Haozhao Li, Yiyan Xu, Wenshuai Jiang for helpful discussions.

\appendix
\section{Dual structure and the proof of (\ref{Bij})}
\subsection{Dual structure determined by circle patterns}\label{Appendix-DDL}
Suppose $C_i,C_j,C_k$ are closed disks centered at $v_i$, $v_j$ and $v_k$ so that their radii are $r_i$, $r_j$ and $r_k$. They both intersect with each other at an angle supplementary to $\Phi(v_iv_j)$, $\Phi(v_jv_k)$, $\Phi(v_kv_i)$. Let $\mathcal{L}_i$, $\mathcal{L}_j$, $\mathcal{L}_k$ be the geodesic lines passing through the pairs of the intersection points of $\{C_k,C_j\}$, $\{C_k,C_i\}$, $\{C_i,C_j\}$. These three lines $\mathcal{L}_i$, $\mathcal{L}_j$, $\mathcal{L}_k$ must intersect in a common point $O$ (see Figure \ref{lij=OH}). Let $i'$ be an intersection point of the circles $C_j$ and $C_k$. Denote
\begin{equation}
l^*_{ij}|_{\triangle ijk}\triangleq r_j \frac{\cos\theta^{i'k}_j-\cos\theta^{ik}_j\cos\theta^{k'i}_j}{\sin\theta^{ik}_j},
\end{equation}
where $\theta^{i'k}_j$ is the inner angle at $v_j$ in the triangle $\{j,k,i'\}$. $l^*_{ij}|_{\triangle ijk}$ is the directed distance from $O$ to $H$, the foot of perpendicular at edge $v_iv_j$. The directed distance from $O$ to edge $v_iv_j$ is positive (negative), when $O$ is inside (outside) $\angle v_iv_jv_k$.

\begin{figure}
\begin{minipage}[t]{0.45\linewidth}
\centering
\includegraphics[width=0.8\textwidth]{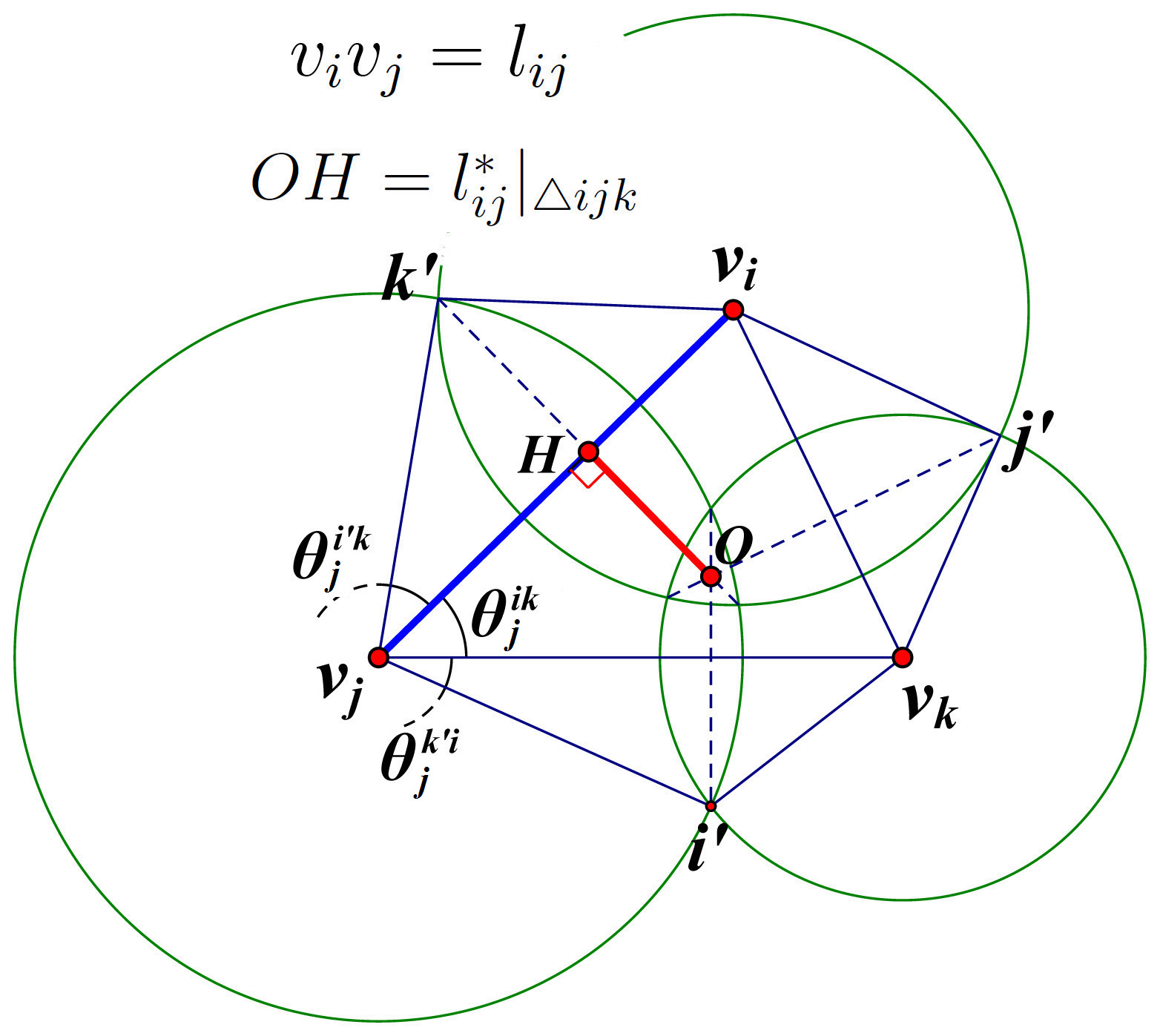}
\caption{discrete dual-Laplacian}
\label{lij=OH}
\end{minipage}
\begin{minipage}[t]{0.45\linewidth}
\centering
\includegraphics[width=0.8\textwidth]{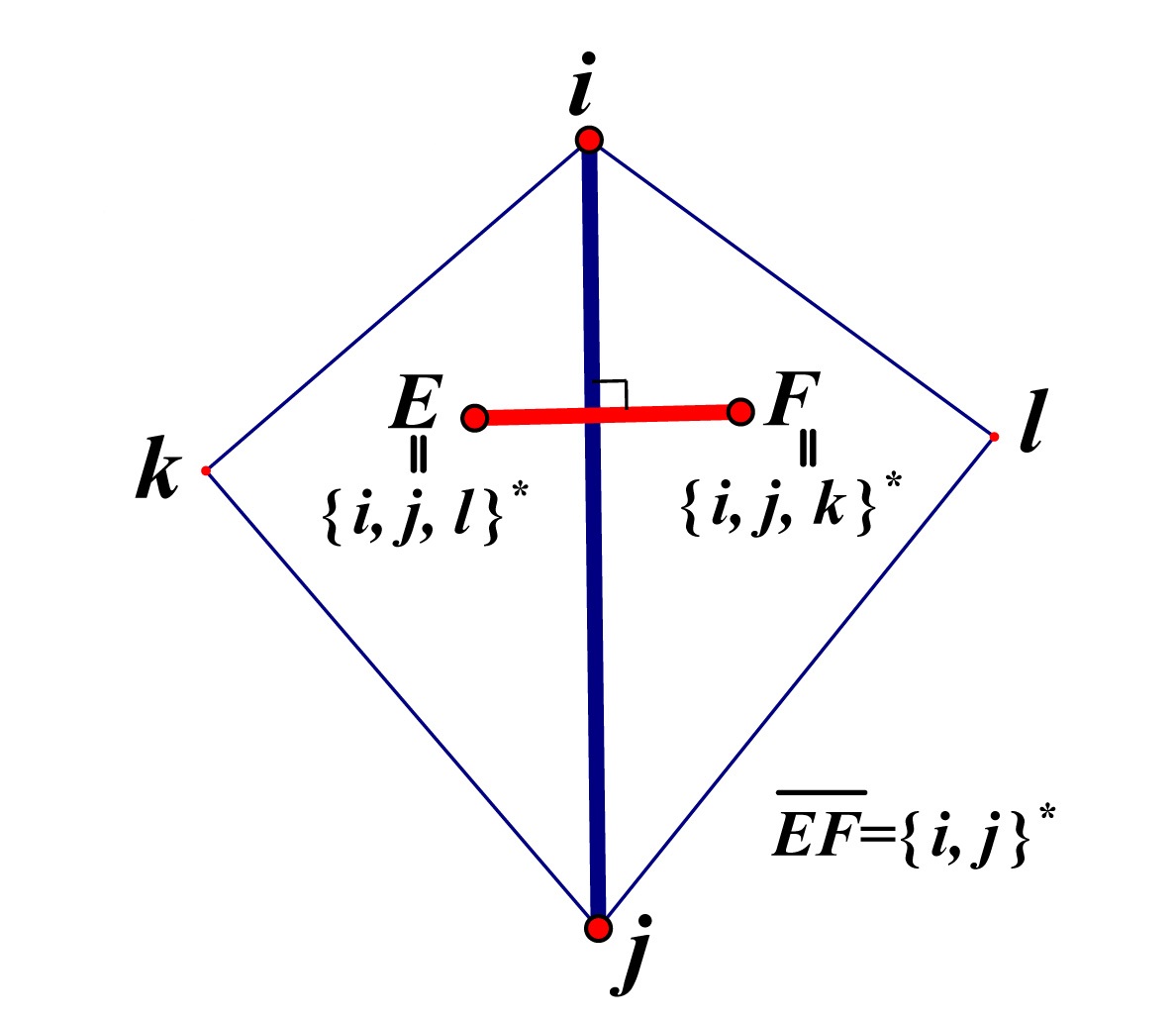}
\caption{dual structure}
\label{duan-config}
\end{minipage}
\end{figure}

Think $O$ as $\{i,j,k\}^*$, the dual vertex to the triangle $\{i,j,k\}$. For two adjacent triangles $\{i,j,k\}$ and $\{i,j,l\}$ having a common edge $\{i,j\}$, $\{i,j\}^*$ (the dual edge to the edge $\{i,j\}$) is the line from $\{i,j,k\}^*$ to $\{i,j,l\}^*$ (see Figure \ref{duan-config}). Note that $\{i,j\}^*$ is perpendicular to $\{i,j\}$. Denote $l_{ij}^*$ as the length of $\{i,j\}^*$, then $l^*_{ij} \doteqdot l^*_{ij}|_{\triangle ijk}+l^*_{ij}|_{\triangle ijl}$. Using cosine law in triangle $\{i,j,k\}$, we have
\begin{equation}\label{half-Bij=lij*}
\frac{\partial \theta^{jk}_i}{\partial r_j} r_j = \frac{r_j}{l_{ij}} \frac{\cos\theta^{i'k}_j-\cos\theta^{ik}_j\cdot \cos\theta^{k'i}_j}{\sin\theta^{ik}_j}= \frac{l^*_{ij}|_{\triangle ijk}}{l_{ij}}.
\end{equation}
Thus we obtain
\begin{equation}
\begin{split}
B_{ij} & =\frac{\partial \theta^{jk}_i}{\partial r_j} r_j+\frac{\partial \theta^{jl}_i}{\partial r_j} r_j= \frac{l^*_{ij}|_{\triangle ijk}}{l_{ij}}+\frac{l^*_{ij}|_{\triangle ijl}}{l_{ij}}= \frac{l^*_{ij}}{l_{ij}}.
\end{split}
\end{equation}
Furthermore,
\begin{equation}\label{Dis-Dual-Lap}
\begin{split}
\Delta f_i & = -\sum_{j=1}^N\frac{\partial K_j}{\partial u_i}f_j = -\sum_{j=1}^NL_{ij}f_j= \:\:\:\:\sum_{j\sim i}B_{ij}(f_j-f_i)=\sum_{j\sim i}\frac{l^*_{ij}}{l_{ij}}(f_j-f_i),
\end{split}
\end{equation}
where $f:V \rightarrow \mathds{R}$ is a function defined on vertices.

The classical discrete Laplace operator $``\Delta"$ is often written as the following form (\cite{CHU})
\begin{equation*}
\Delta f_i=\sum_{j\sim i}\omega_{ij}(f_j-f_i),
\end{equation*}
where the weight $\omega_{ij}$ can be arbitrarily selected for different purpose. Thurston first studied the dual structure of circle patterns in \cite{T1}. Bennett Chow and Feng Luo calculated formula (\ref{half-Bij=lij*}) in \cite{CL1}. Glickenstein wrote $\omega_{ij}=\frac{l^*_{ij}}{l_{ij}}$ in \cite{G1}. The weight $\omega_{ij}$ comes from dual structure of circle patterns. This is why we call (\ref{Dis-Dual-Lap}) discrete dual-Laplacian. There are other types of discrete Laplacians, such as cotangent-Laplacian. See more in \cite{G3}, \cite{G4}, \cite{G5}, \cite{He}, and \cite{Hi}.

\subsection{proof of (\ref{Bij})}
\begin{lemma}\label{zhangruixiang}
Consider a triangle $\triangle ABC$ coming from a configuration of circle patterns with weight $\Phi\in [0,\frac{\pi}{2}]$. Assuming $O$ is inside $\triangle ABC$. $OH$ is the altitude from $O$ onto side $BC$. Assuming the point $H$ lies between $B$ and $C$ (see Figure \ref{figure-zhangruixiang}). Then $|OH|<\sqrt{3}|BC|$.
\end{lemma}

\begin{figure}
\centering
\includegraphics[width=0.36\textwidth]{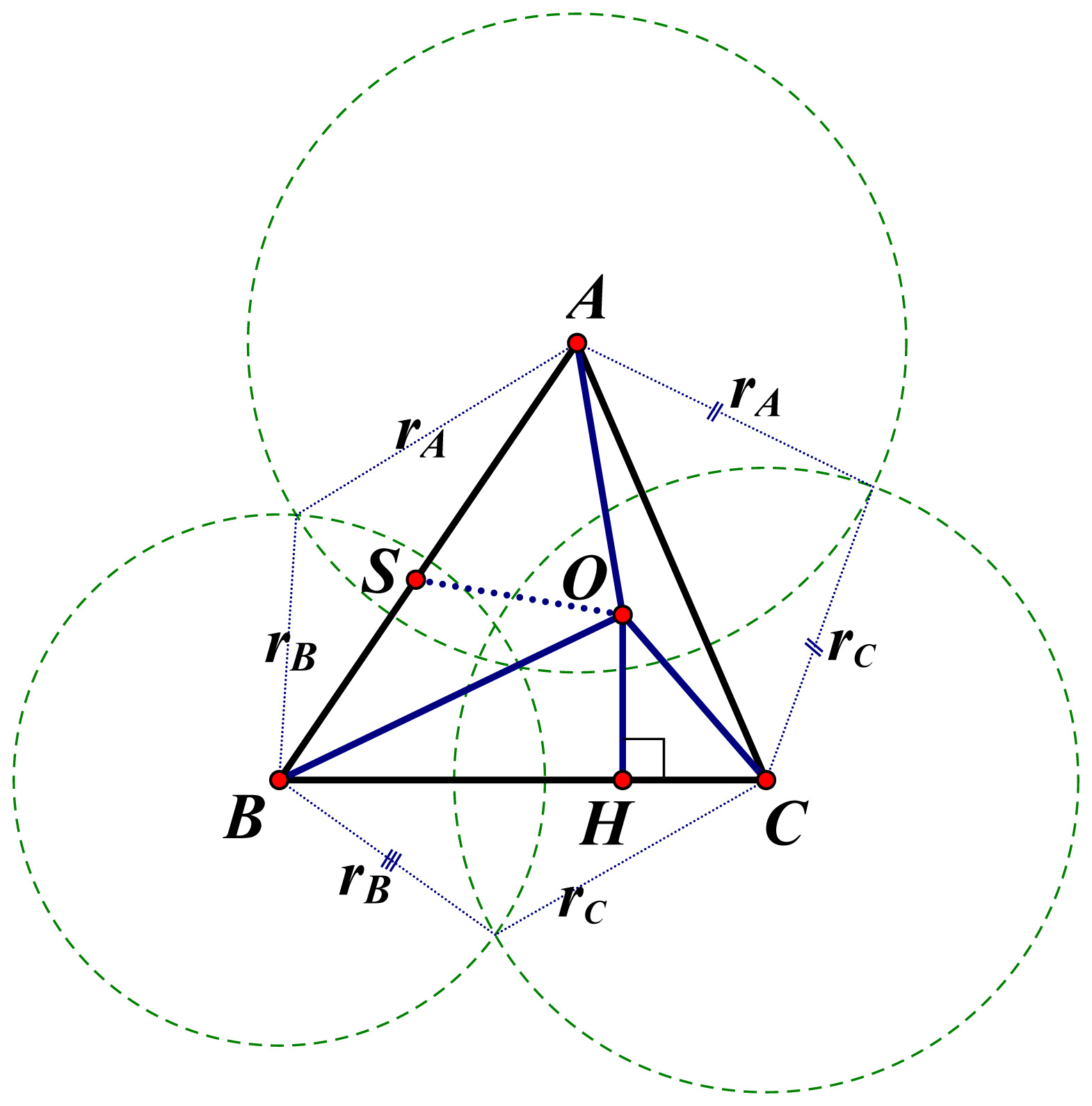}
\caption{O is inside $\triangle ABC$, such that $H$ lies between $B$ and $C$.}
\label{figure-zhangruixiang}
\end{figure}

\begin{proof}
Suppose this lemma is false. Then $|OH|\geq \sqrt{3}|BC|$, and
\begin{displaymath}
|OH|\geq \sqrt{3}|BC| \Rightarrow
\begin{cases}
|OH|\geq \sqrt{3}|BH|\Rightarrow \measuredangle BOH\leq \pi/6,  \\[7pt]
|OH|\geq \sqrt{3}|CH|\Rightarrow \measuredangle COH\leq \pi/6.
\end{cases}
\end{displaymath}
Hence
$$\measuredangle BOC=\measuredangle BOH + \measuredangle COH \leq \pi/3,$$
and
$$\measuredangle AOC=2\pi-(\measuredangle AOB+\measuredangle BOC)\geq 2\pi-\left(\pi+\frac{\pi}{3}\right)=\frac{2\pi}{3},$$
$$\measuredangle AOB=2\pi-(\measuredangle AOC+\measuredangle BOC)\geq 2\pi-\left(\pi+\frac{\pi}{3}\right)=\frac{2\pi}{3}.$$
In $\triangle AOB$, we have $|AB|>|AO|$, for $AB$ faces to a much bigger angle. Thus we can select a unique point $S$ between $A$ and $B$ so that $|AS|=|AO|$. using
$$\frac{\measuredangle AOB}{2}<\frac{\pi-\measuredangle BAO}{2}<\frac{\pi}{2},$$
we get
\begin{displaymath}
\measuredangle SOB = \measuredangle AOB -\frac{\pi-\measuredangle BAO}{2}\in \left[\measuredangle AOB -\frac{\pi}{2},\,\measuredangle AOB -\frac{\measuredangle AOB}{2}\right]\subset \left[\frac{\pi}{6},\frac{\pi}{2}\right],
\end{displaymath} and
\begin{displaymath}
\measuredangle OSB = \pi-\frac{\pi-\measuredangle BAO}{2}\in \left[\pi -\frac{\pi}{2},\,\pi-\frac{\measuredangle AOB}{2}\right]\subset \left[\frac{\pi}{2},\frac{2\pi}{3}\right].
\end{displaymath}
Then we obtain
$$\frac{|OB|}{|AB|-|AO|}=\frac{|OB|}{|BS|}=\frac{\sin \measuredangle BSO}{\sin \measuredangle SOB}\leq\frac{1}{(\frac{1}{2})}=2,$$
which implies $$|AB|-|AO|\geq\frac{1}{2}|OB|.$$
Similarly, we have
$$|AC|-|AO|\geq\frac{1}{2}|OC|.$$
Next we show
$$r_B\geq\frac{1}{2}|OB|$$
and $$r_C\geq\frac{1}{2}|OC|.$$
If $r_B\leq|OB|$, we know $r_A\leq|OA|$ by use of $r^2_A-r^2_B=|OA|^2-|OB|^2$. We already know $r_A+r_B\geq|AB|$. Hence
$$r_B\geq|AB|-r_A\geq|AB|-|OA|\geq\frac{1}{2}|OB|.$$
Thus we always have $r_B\geq\frac{1}{2}|OB|$, no matter $r_B\geq |OB|$ or $r_B \leq|OB|$.
Similarly, we have $r_C\geq\frac{1}{2}|OC|$.
Then it is easy to see
\begin{displaymath}
\begin{split}
r^2_B+r^2_C & \geq \frac{1}{4}(|OB|^2+|OC|^2)\geq \frac{1}{4}\cdot2|OH|^2 = \frac{1}{2}|OH|^2 \geq \frac{3}{2}|BC|^2> |BC|^2,
\end{split}
\end{displaymath}
which contradicts with the fact
$$|BC|=\sqrt{r^2_B+r^2_C+2r_Br_C\cos\Phi_{BC}}\geq \sqrt{r^2_B+r^2_C}.$$
Thus the lemma is true.\qed
\end{proof}

(Above Lemma belongs to Ruixiang Zhang and Chenjie Fan.)

\begin{corollary} Given $(X,T,\Phi)$, where $X$ is a closed surface, $T$ is a triangulation, $\Phi\in[0,\pi/2]$ is a weight. Then
\begin{equation*}
 0<\frac{\partial \theta^{jk}_i}{\partial r_j}r_j<\sqrt{3}.
\end{equation*}
\end{corollary}
\begin{proof}
Thurston claimed that $\frac{\partial \theta^{jk}_i}{\partial r_j}r_j=l^*_{ij}|_{\triangle ijk}$ is positive (\cite{T1}). Hence $O$ is inside the $\triangle v_iv_jv_k$. The proof can be finished by Lemma \ref{zhangruixiang}.
\end{proof} \qed

\section{Combinatorial Ricci potential is proper}
\begin{lemma} \label{proper}
Assuming function $\psi\in C^2(\mathds{R}^n)$ satisfies
\begin{description}
  \item[(1)] $\psi$ is strictly convex to the downwards, and
  \item[(2)] there exists at least one point $p$ such that $\nabla \psi(p)=0$.
\end{description}
Then $\lim\limits_{x\rightarrow\infty}\psi(x)=+\infty$. Moreover, $\psi$ is proper.
\end{lemma}
\begin{proof}
Set $h(t)\triangleq \underset{|x|=t}{\,\inf}\,\,\psi(x)$, $t\geq 0$. Then $h(t)$ is a strictly monotone increasing function. We just need to prove $h(t)\rightarrow +\infty$, as $t\rightarrow +\infty$. The process is almost the same with Lemma B.1 in \cite{Ge2}, we omit the details. Notice that there is another method to prove Lemma \ref{proper} by remark 2.7 in \cite{BO1}. \qed
\end{proof}

\begin{theorem} \label{Ricci-potential-conveg-infty}
Given $(X,T,\Phi)$, where $X$, $T$ and $\Phi$ are defined as before. Assuming there exists a constant curvature metric $r_{av}\in\mathds{R}_{>0}^N$. Consider the combinatorial Ricci potential
\begin{equation*}
f(u)\triangleq\int_{u_{av}}^u\sum_{i=1}^N\big(K_i-k_{av}\big)du_i\,,\,\,u\in \mathds{R}^N,
\end{equation*}
where $u_{av}=Exp^{-1}(r_{av})$. Then for arbitrary constant $a$, $f\big|_{\mathscr{U}_a}$ is proper and
$$\lim\limits_{u\rightarrow \infty,\,u\in\mathscr{U}_a}f(u)=+\infty.$$
\end{theorem}
\begin{proof}
Write $n=N-1$ for convenience, and denote $\alpha^T=\frac{1}{\sqrt{N}}(1,\cdot\cdot\cdot,1)^T$ as the unit normal vector perpendicular to the hyperplane $\mathscr{U}_a$.
Then we have $f(u)=f(u+t\alpha^T)$ for all $t\in \mathds{R}$. Thus we only need to prove above theorem for $a=0$. Write $\mathscr{U}_0$ as $\mathscr{U}$ for short.
The hyperplane $\mathscr{U}$ is determined by $n$ linear independent variables.
We claim that $f\big|_{\mathscr{U}}$ is strictly convex when considered as a function of $n$ independent variables.

Denote $\mathds{R}^n_{\zeta}\triangleq\{\zeta=(\zeta_1,\cdots,\zeta_N)^T\in\mathds{R}^N|\zeta_N=0)\}$. Select an orthogonal linear transformation $\mathscr{A}:\mathds{R}^N\rightarrow\mathds{R}^N$, such that $\mathscr{A}\alpha^T=(0,\cdots,0,1)^T$ and, moreover, $\mathscr{A}$ transforms $\mathscr{U}$ to $\mathds{R}^n_{\zeta}$ one by one. The matrix form of $\mathscr{A}$ under standard orthogonal basis of $\mathds{R}^N$ is denoted by $A$.
Set $g=f\big|_{\mathscr{U}}\circ \mathscr{A}^{-1}\big|_{\mathds{R}^n_{\zeta}} $, \emph{i.e.},
$$g(\zeta_1,\cdots,\zeta_n)\triangleq f\Big(A^T(\zeta_1,\cdots,\zeta_n,0\big)^T\Big).$$
We want to show that $Hess(g)$ is a $n\times n$ positive definite matrix. Partition the matrix $A$ into two blocks
\begin{displaymath}
A=
\begin{bmatrix}
S \\ \alpha_N
\end{bmatrix},
\end{displaymath}
where $S$ is a $n\times N$ matrix.
By calculation we get $\nabla f=K-K_{av}$, $Hess(f)=\frac{\partial(K_1,...,K_N)}{\partial(u_1,...,n_N)}=L$, $\nabla g=S \nabla f=S(K-K_{av})$ and $Hess(g)=SLS^T$.
Because $A\alpha^T=(0,\cdots,0,1)^T$, we have $\alpha=(0,\cdots,0,1)A=\alpha_N$. Using $L\alpha^T=0$, we obtain
\begin{displaymath}
ALA^T=
\begin{bmatrix}
S \\ \alpha
\end{bmatrix}
L\left[S^T,\alpha^T\right]=
\begin{bmatrix}
SLS^T & SL\alpha^T \\
\alpha LS^T & \alpha L\alpha^T
\end{bmatrix}
=
\begin{bmatrix}
SLS^T & 0\,\,\\
0 & 0\,\,
\end{bmatrix}.
\end{displaymath}
On one hand, $SLS^T$ is positive semi-definite, since $ALA^T$ is positive semi-definite. On the other hand, $rank(SLS^T)=rank(ALA^T)=rank(L)=n$. Therefore $Hess(g)=SLS^T$ is positive definite. Moreover, $f\big|_{\mathscr{U}} (=g\circ \mathscr{A}\big|_{\mathscr{U}})$ is strictly convex when considered as a function of $N-1$ variables. Thus we get the claim above. The proof can be finished by using Lemma \ref{proper}.
\end{proof}
\qed

Huabin Ge

School of Mathematical Sciences, Peking Univ., Beijing 100871, PR China

Email: gehuabin@pku.edu.cn

\end{document}